\documentclass[twoside,leqno,10pt]{amsart}
\usepackage{amsfonts}
\usepackage{amsmath}
\usepackage{amscd}
\usepackage{amssymb}
\usepackage{amsthm}
\usepackage{latexsym}
\usepackage{color}
\newtheorem{theorem}{Theorem}

\newtheorem{Proposition}{Proposition}

\newtheorem{remark}{Remark}

\numberwithin{equation}{section}

\title[evaluation of Exponential sums and small solutions of quadratic forms]{Multiple exponential sums and their applications to quadratic congruences}
\author{Nilanjan Bag \and  Stephan Baier \and Anup Haldar} 
\address{Nilanjan Bag,
Thapar Institute of Engineering and Technology, Department of Mathematics, Patiala, Punjab, 147004, India}
\email{nilanjanb2011@gmail.com}
\address{Stephan Baier,
Ramakrishna Mission Vivekananda Educational and Research Institute, Department of Mathematics, G. T. Road, PO Belur Math, Howrah, West Bengal 711202, India}
\email{stephanbaier2017@gmail.com}
\address{Anup Haldar,
Ramakrishna Mission Vivekananda Educational and Research Institute, Department of Mathematics, G. T. Road, PO Belur Math, Howrah, West Bengal 711202, India}
\email{anuphaldar1996@gmail.com}

\begin{document}
\begin{abstract}
In this paper, we develop a method of evaluating general exponential sums with rational amplitude functions for multiple variables which complements work by T. Cochrane and Z. Zheng on the single variable case. As an application, for $n\geq 3$, a fixed natural number, we obtain an asymptotic formula for the (weighted) number of solutions of general quadratic congruences modulo $p^m$ in small boxes, thus establishing an equidistribution result for these solutions. 
\end{abstract}
\maketitle
\tableofcontents

\section{Introduction and main results} 
In this paper, we develop a method for evaluating and estimating general exponential sums with rational amplitude functions for multiple variables. A general pure exponential sum (i.e., without twist by multiplicative characters) in $n$ variables is of the type
\begin{align*}
\mathcal{S}(f,p^m)=\sum_{\substack{{\bf{x}}\in \left(\mathbb{Z}_p/{p^m\mathbb{Z}_p}\right)^n}}e_{p^m}(f({\bf{x}})),
\end{align*}
where $f=f_1/f_2$ is a rational function with integer coefficients such that the reductions of $f_1,f_2\in \mathbb{Z}[x_1,...,x_n]$ modulo $p$ are coprime, $m$ is a natural number, $p$ is a prime and $e_{q}(\cdot)$ is the additive character
\begin{align*}
e_{q}(x)=e^{2\pi i x/q}.
\end{align*} 
Here it is understood that we sum only over ${\bf x}$ such that $f_2({\bf x})\not=0 \bmod{p}$ so that $f({\bf x})$ is well-defined. 
These sums generalize classical pure exponential sums 
\begin{align*}
S(f,p^m)=\sum_{x=1}^{p^m}e_{p^m}(f(x))
\end{align*}
in one variable.
For $m=1$ and polynomials $f\in\mathbb{Z}[x]$, from the work of Weil \cite{weil} on the Riemann Hypothesis for curves over a finite field, it is well known that if the degree of $f\bmod p$, denoted by $d_p(f)$, is greater or equal to $1$, then 
\begin{align*}
|S(f,p)|\leq d_p(f)p^{1/2}.
\end{align*}
 Bombieri \cite[Theorem 5]{BM} sharpened this to a uniform upper bound which holds for rational functions $f$ in one variable. Later Perel'muter \cite{PM} extended Bombieri's result to mixed exponential sums, i.e., to exponential sums twisted with a multiplicative character.
 For $m\ge 2$, there are numerous results on these sums in the literature, mainly for the case of polynomials (see \cite{JRC, JHHC, JHL1, JHL2, IS}, for example). A famous general upper bound for pure exponential sums in one variable is due to Hua \cite{HUA1} who established that
 \begin{align*}
 |S(f,p^m)|\leq Cp^{m(1-1/d)}
 \end{align*}
 for any polynomial $f\in\mathbb{Z}[x]$ of degree $d$ and a constant $C$ depending on $d$, which has been refined in several works. More recently, for $f\in \mathbb{Z}[x]$ and $m\ge 2$, Cochrane and Zheng \cite{CoZ} applied stationary phase-type arguments to derive closed formulas for pure and mixed exponential sums with modulus $p^m$ in certain cases and sharpened upper bounds in general. This has been extended to rational functions $f$ in  \cite{Coc2} (see, in particular, \cite[Theorem 1.1]{Coc2}). In this article, we establish a multivariable version of Cochrane's result for pure exponential sums. These sums were also considered by B. Fisher in \cite{Fisher}, where he adapted Cochrane's methods to obtain closed formulas and upper bounds for these multivariable sums. However, the conditions in his results become rather complicated. Here we present simpler versions which are easier to use for applications. Before we discuss our applications, we will set up the required notations and formulate our results on exponential sums. 

 \subsection{Notations}
Throughout this article, we will use the following notations. 
 \begin{itemize}
\item[1.] We assume that $p$ is an odd prime, $m\ge 2$ is an integer and $f$ is a rational function of the form $f_1/f_2$, where $f_1,f_2\in\mathbb{Z}[x_1,...,x_n]$ are such that the reductions of $f_1$ and $f_2$ modulo $p$ are coprime.

 \item[2.] The gradient of the function $f$ at a point ${\bf x}$, denoted as $\nabla f({\bf x})$, is defined by 
\begin{align*}
        \nabla f({\bf x})=\begin{bmatrix}
             \frac{\partial f}{\partial x_1}({\bf x})\\
              \frac{\partial f}{\partial x_2}({\bf x})\\
              \vdots\\
               \frac{\partial f}{\partial x_n}({\bf x})
        \end{bmatrix}.
\end{align*}

 \item[3.] The Hessian of a function $f$, denoted by Hesse$(f)$, is defined as 
\begin{align*}
    \text{Hesse}(f)=\begin{bmatrix}
       \frac{\partial^2 f}{\partial x_1^2} &  \frac{\partial^2 f}{\partial x_1\partial x_2}& \cdots &\frac{\partial^2 f}{\partial x_1 \partial x_n} \\
       \frac{\partial^2 f}{\partial x_2x_1} &  \frac{\partial^2 f}{\partial x_2^2}&\cdots& \frac{\partial^2 f}{\partial x_2 \partial x_n}\\
       \vdots&\vdots&\cdots&\vdots\\
       \frac{\partial^2 f}{\partial x_n \partial x_1} &  \frac{\partial^2 f}{\partial x_n\partial x_2}& \cdots &\frac{\partial^2 f}{\partial x_n^2}
       \end{bmatrix}.
\end{align*}
 
\item[4.]  For a prime $p$, we define the order of $f$ as 
\begin{align*}
    \text{\rm ord}_p(f)=\text{\rm ord}_p(f_1)-\text{\rm ord}_p(f_2),
\end{align*}
where 
\begin{align*}
    \text{\rm ord}_p(f_i)=\max\{k\in\mathbb{Z}_{\geq 0}|~p^k ~\text{divides all coefficients of}~ f_i\}.
\end{align*}

\item[5.] If $r=\min_{1\leq i\leq n}\{\text{\rm ord}_p(\frac{\partial f}{\partial x_i})\}$, we denote by $\mathcal{C}_f$ the set of critical points defined as 
\begin{align*}
\mathcal{C}_f:=\{{\bf{x}}\in \mathbb{Z}_p^n| p^{-r}\nabla f({\bf x})={\bf{0}}\bmod p\}.
\end{align*}

\item[6.] We denote by $G_q(n)$ the quadratic Gauss sum
$$
G_q(n)=\sum\limits_{x=1}^q e_q(nx^2).
$$

\item[7.] For odd $q\in \mathbb{N}$, let $\left(\frac{\bullet}{q}\right)$ denote the Legendre symbol.

\item[8.] We denote by $\mathbb{Z}_p$ the set of $p$-adic integers, which is represented in the form
\begin{align*}
\mathbb{Z}_p=\{a_0+a_1p+\cdots+a_ip^i+\cdots|~ a_j\in \{0,1,...,p-1\} \mbox{ for all } j\in \mathbb{N}_0\}.
\end{align*}
 
\item[9.] For any $n$-tuple ${\alpha}\in \mathbb{Z}^n$, we define
\begin{align*}
S_{\bf{\alpha}}(f,p^m)=\sum_{\substack{{\bf{x}}\in \left(\mathbb{Z}_p/{p^m\mathbb{Z}_p}\right)^n\\{\bf{x}}\equiv \alpha \bmod p}}e_{p^m}(f({\bf{x}})),
\end{align*}
with the convention that the sum runs over those ${\bf x}$ for which the denominator of $f=f_1/f_2$ satisfies $f_2({\bf x})\not\equiv 0 \bmod{p}$, i.e. $f({\bf x})$ is well-defined modulo $p^m$. 
\item[10.] Throughout the article $\overline{x}$ denotes the multiplicative inverse of $x$ to relevant modulus.
\end{itemize}

We first establish the following evaluation of $S_{{\bf \alpha}}(f,p^m)$.
\begin{Proposition}\label{P1}
Let $p$ be an odd prime and $f$ be a non-constant rational function in $n$ variables defined over $\mathbb{Z}$. Set $r:=\min_{1\leq i\leq n}\{\text{\rm ord}_p(\frac{\partial f}{\partial x_i})\}$. Then we have 
\begin{equation} \label{basiceq}
\begin{split}
S_{\bf{\alpha}}(f,p^m)=\begin{cases}p^{n(m+r)/2}\displaystyle\sum_{\substack{{\bf{y}}\in(\mathbb{Z}_p/p^{(m-r)/2}\mathbb{Z}_{p})^n\\{\bf{y}}\equiv {\bf{\alpha}}\bmod p\\p^{-r}\nabla f({\bf{y}})\equiv 0\bmod p^{(m-r)/2}}}e_{p^{m}}(f({\bf{y}})),~&m-r~\text{is even};\\
p^{(m+r-1)n/2}\displaystyle\sum_{\substack{{\bf{y}}\in(\mathbb{Z}_p/p^{(m-r+1)/2}\mathbb{Z}_{p})^n\\{\bf{y}}\equiv {\bf{\alpha}}\bmod p\\p^{-r}\nabla f({\bf{y}})\equiv 0\bmod p^{(m-r-1)/2}}}e_{p^{m}}(f({\bf{y}})),~&m-r~\text{is odd}.
\end{cases}
\end{split}
\end{equation}
\end{Proposition}
This will give rise to our first main result below. 
\begin{theorem} \label{MT1}
Let $p$ be an odd prime and $f$ be a non-constant rational function in $n$ variables defined over $\mathbb{Z}$. Let $m\in \mathbb{N}$, define $r$ as above and assume that $m-r\ge 2$. Then we have the following.
\begin{itemize}
\item[(i)] If ${\bf{\alpha}}\not\in \mathcal{C}_f$, then 
\begin{align*}
S_{\bf{\alpha}}(f,p^m)=0.
\end{align*}

\item[(ii)] If ${\bf{\alpha}}\in \mathcal{C}_{f}$ and {\rm det(Hesse}$(f)({\bf{\alpha}}))\not\equiv 0 \bmod{p}$, then 
\begin{align*}
S_{\bf{\alpha}}(f,p^m)=\begin{cases}p^{n(m+r)/2}e_{p^m}(f({\bf{{\bf{\alpha}^{*}}}})),~&m-r~\text{\rm is even};\\
\left(\frac{\text{\rm det}(\overline{2}p^{-r}\text{\rm Hesse}(f)({\bf{\alpha}^*}))}{p}\right)p^{(m+r-1)n/2}e_{p^{m}}(f({\bf{\alpha}^*}))G_p(1)^n,~&m-r~\text{\rm is odd},
\end{cases}
\end{align*}
where $\overline{2}$ is a multiplicative inverse of $2$ modulo $p^m$ and $\bf{\alpha}^{*}$ is the unique lifting of a solution $\alpha$ of the congruence $p^{-r}\nabla f({\bf x})={\bf{0}}\bmod p$ to a solution of the congruence $p^{-r}\nabla f({\bf{x}})\equiv 0 \bmod p^{m}$.
\end{itemize}
\end{theorem}
%
%

The case of critical points at which the Hesse matrix is singular (i.e., {\rm det(Hesse}$(f)({\bf{\alpha}}))\equiv 0 \bmod{p}$) requires more refined arguments and will not be considered in this paper. This may be subject to future research. 

Since complete exponential sums with rational functions appear frequently in analytic number theory, it may be expected  that Theorem 1 can be applied to many situations. In this article, we will employ it to derive an asymptotic formula for the number of solutions of special quadratic congruences in several variables modulo prime powers in small boxes. Here the relevant exponential sums are quadratic Gauss sums in several variables with certain restrictions on the summation variables. Closely related exponential sums were evaluated explicitly in \cite[sections 10]{HB-circle} for prime and prime square moduli. We infer our evaluations for arbitrary prime power moduli in a straight-forward way from the much more general Theorem 1 above.    
In the following, we give some motivation for the problem under consideration. 

Let $Q(x_1,...,x_n)$ be a quadratic form with integer coefficients. The question of detecting small solutions of congruences of the form
\begin{equation} \label{ourcongruence}
Q(x_1,...,x_n)\equiv 0 \bmod{q}
\end{equation}
has received a lot of attention (see, in particular, \cite{BH1}, \cite{BH-UP}, \cite{Hak}, \cite{AH1},
\cite{Hea1}, \cite{Hea2}, \cite{Hea3}, \cite{SSS}).  For the case $n=3$, a result by Schinzel, Schlickewei and Schmidt \cite{SSS} implies that there is always a non-zero solution $(x_1,x_2,x_3)\in \mathbb{Z}^3$ to \eqref{ourcongruence} such that $\max\{|x_1|,|x_2|,|x_3|\}=O(q^{2/3})$, where the $O$-constant is absolute.  For $q$ square-free, the exponent $2/3$ was improved to $5/8$ by Heath-Brown \cite{Hea3}. Of particular interest is the complementary case when $q=p^m$ is a power of a fixed prime $p$ and $m$ tends to infinity. This was considered by Hakimi \cite{Hak} with special emphasis on forms with a large number of variables, as compared to the exponent $m$. For diagonal forms 
\begin{align}\label{Q-form}
Q_{{\bf{\alpha}}}(x_1,...,x_n)=\alpha_1x_1^2+\cdots +\alpha_nx_n^2
\end{align}
modulo prime powers $q=p^m$ with an {\it arbitrary} sequence $\alpha=(\alpha_1,...,\alpha_n)$ of coefficients (i.e. $\alpha$ not depending on $m$), the second- and third-named authors \cite{BH1} improved the said exponent $2/3$ to $11/18$ under the stronger condition $(x_1\cdots x_n,p)=1$ on the solutions $(x_1,...,x_n)$ of \eqref{ourcongruence} (see \cite[Theorem 2]{BH1}). A result by Cochrane \cite{Coc} for general moduli $q$ implies that for any form $Q(x)$ with ${\it fixed}$ coefficients, there is a non-zero solution with $\max\{|x_1|,|x_2|,|x_3|\}=O(q^{1/2})$, where the $O$-constant may depend on the form. In the above-mentioned paper \cite{BH1}, the second- and third-named authors refined this result into an asymptotic formula for the case of fixed diagonal forms modulo prime powers (see \cite[Theorem 1]{BH1}). This was generalized by the third-named author in \cite{AH1} to general quadratic forms $ax^2+bxy+cy^2+dxz+eyz+fz^2$ with fixed coefficients. 

In the present paper, we are interested in asymptotic formulas for the number of solutions $(x_1,...,x_{n})$ of the congruence 
\begin{equation} \label{c}
Q(x_1,x_2,...,x_{n})\equiv 0 \bmod p^m
\end{equation}
for a general quadratic form
\begin{align*}
 Q(x_1,x_2,...,x_{n})= \sum_{i,j=1}^n a_{ij}x_ix_j, ~a_{ij}\in\mathbb{Z}.
\end{align*}
Here we assume that $A=(a_{ij})_{1\le i,j\le n}$ is a symmetric matrix such that $(\det(A),p)=1$, i.e. $Q$ is non-singular modulo $p$.   
 We count weighted solutions in boxes $I_1\times \cdots\times I_{n}$, where $I_1,...,I_{n}$ are intervals of length $2N$, which we aim to take as small as possible compared to the modulus. 
The main novelties in this paper, as compared to the previous works \cite{BH1} and \cite{AH1}, are two-fold. Firstly, we now use evaluations of {\it multivariable} exponential sums, as provided in Proposition \ref{P1} and Theorem \ref{MT1}. Secondly, we infer our results directly from estimates for these exponential sums, avoiding a parametrization of solutions of quadratic congruences which was used in \cite{BH1} and \cite{AH1}. This simplifies our evaluations a lot, and the method becomes more elegant.

The asymptotic behaviour for the number of solutions may change drastically around the point $N=q^{1/2}=p^{m/2}$. If $N\ge q^{1/2+\varepsilon}$, we will obtain an asymptotic for the number of solutions in question of the form $\sim b_p\cdot N^{n}/q$ for a certain constant $b_p$ depending on $p$. However, if $N\le q^{1/2-\varepsilon}$ and our box is centered as ${\bf 0}$, then the congruence \eqref{c}
collapses into an equation
$$
Q(x_1,x_2,...,x_{n})= 0
$$ 
since in this case the squares of the variables are much smaller than the modulus $q$. Here  the number of solutions satisfies an asymptotic of the form $\sim c N^{n-2}(\log N)^{r_n}$, where $c$ is a product of local densities and $r_n=1$ if $n\le 4$ and $r_n=0$ if $n>4$. In particular, in the case $n=2$, we count Pythagorean triples. 
We point out that the case of Pythagorean triples modulo prime powers was considered in the unpublished work \cite{BH-UP}.
In \cite{AH1} the third-named author studied the general quadratic form for $n=3$ with some restrictions to the coefficients of the form. In this paper, we study arbitrary non-singular quadratic from for $n$ variables without any such additional restrictions. We establish the following. 
\begin{theorem} \label{MT2}
Assume that $A=(a_{ij})_{1\le i,j\le n}$ is a symmetric matrix with integer coefficients such that $(\det(A),p)=1$, i.e. $Q$ is non-singular modulo $p$.   
Let $\varepsilon>0$ and $n(\geq 3)\in\mathbb{N}$ be fixed, $\Phi:\mathbb{R}\rightarrow \mathbb{R}_{\ge 0}$ be a Schwartz class function and $p$ be an odd prime. Let $(x_{0,1},x_{0,2},...,x_{0,n})$ be a fixed point in $\mathbb{Z}^{n}$. Then as $m\rightarrow \infty$, we have the asymptotic formula
\begin{equation} \label{main}
\sum\limits_{\substack{(x_1,...,x_{n})\in \mathbb{Z}^{n}\\ (x_1,...,x_n)\not\equiv (0,...,0)\bmod p\\\sum_{i,j=1}^n a_{ij}x_ix_j\equiv 0 \bmod{p^m}\\
}}
\prod_{i=1}^{n}\Phi\left(\frac{x_i-x_{0,i}}{N}\right)\sim b_p\cdot
\hat{\Phi}(0)^{n}\cdot \frac{N^{n}}{p^m},
\end{equation}
if $N\ge  p^{(1/2+\varepsilon)m}$, where $b_p=\#\mathcal{A}/p^{m(n-1)}$ with
\begin{align*}
 \mathcal{A}&=\left\{(x_1,...,x_{n})\in\mathbb{Z}^n|~1\leq x_1,...,x_{n}\leq p^m,~  (x_1,...,x_n)\not\equiv (0,...,0) \bmod p, \ \sum_{i,j=1}^n a_{ij}x_ix_j\equiv 0 \bmod p^m\right\}.
    \end{align*}
\end{theorem}
\begin{remark}
Note that $b_p$ is independent of $m$. We have
$$
\sharp\mathcal{A}=p^{(n-1)(m-1)}\sharp \left\{(x_1,...,x_{n})\in\mathbb{Z}^n|~1\leq x_1,...,x_{n}\leq p,~  (x_1,...,x_n)\not\equiv (0,...,0) \bmod p, \ \sum_{i,j=1}^n a_{ij}x_ix_j\equiv 0 \bmod p\right\},
$$
which can be shown using a Hensel's type argument. To elaborate more, assume that  $(x_1+k_1p^k,...,x_n+k_np^k)$ is a solution of $Q(x_1,...,x_n)=\sum_{i,j=1}^n a_{ij}x_ix_j$ modulo $p^{k+1}$, where $(x_1,...,x_k)$ is a solution modulo $p^k$ and $0\leq k_i\leq p-1$. Using Taylor expansion, $Q(x_1+k_1p^k,...,x_n+k_np^k)=Q(x_1,...,x_n)+p^k(k_1,...,k_n)\cdot \nabla Q(x_1,...,x_n)\bmod p^{k+1}$. This gives 
\begin{align*}
   (k_1,...,k_n)\cdot\nabla Q(x_1,...,x_n)\equiv -  \frac{Q(x_1,...,x_n)}{p^k}\bmod p,
\end{align*}
which has $p^{n-1}$ number of solutions in $(k_1,...,k_n)$. This proves our claim by lifting up to modulo $p^m$. We have $b_p>0$ as a consequence of \cite[subsection 1.7, Proposition 4]{Serre}. 
\end{remark}
%
Our proof of Theorem \ref{MT2} above follows the main lines in \cite{BH1} and \cite{AH1}, but we proceed without using parametrization. Instead
we directly apply Poisson summation and evaluate the resulting exponential sums explicitly
to prove our result. \\ \\
{\bf Acknowledgements.} The authors would like to thank the anonymous referees for their valuable comments. They would also like to thank the Ramakrishna Mission Vivekananda Educational and Research Institute for providing excellent working conditions. A.H. would like to thank Professor S.D. Adhikari for his suggestion to consider quadratic forms of $n$ variables. 
During the preparation of this article, N.B. was supported by the National Board
of Higher Mathematics post-doctoral fellowship (No.: 0204/3/2021/R\&D-II/7363)
and A.H. would like to thank CSIR, Government of India for financial support in the form of a Senior Research Fellowship (No.: 09/934(0016)/2019-EMR-I).\\ \\
{\bf Data availability statement.}  Data sharing is not applicable to this article as no new data were created or analyzed in this study.\\ \\
{\bf Conflict of interest statement.} The authors declare no conflicts of interest regarding this manuscript.

\section{Preliminaries}
The following preliminaries will be needed in the course of this paper. 
Let $q>0$ be an odd integer. In this case, whenever $(n,q)=1$, we have 
\begin{align}\label{gauss-1}
G_q(n)=\left(\frac{n}{q}\right)G_q(1),
\end{align}
where $\left(\frac{n}{q}\right)$ is the Jacobi symbol. 

The next required tool is the Poisson summation formula.

\begin{Proposition}[Poisson summation formula] \label{Poisson} Let $\Phi : \mathbb{R}\rightarrow \mathbb{R}$ be a Schwartz class function, $\hat\Phi$ its Fourier transform and $P>0$, $x$ be real numbers. Then
$$
\sum\limits_{n\in \mathbb{Z}} \Phi(x+nP)=\sum\limits_{n\in \mathbb{Z}} \frac{1}{P}\hat\Phi\left(\frac{n}{P}\right)\cdot e_P(nx).
$$

\end{Proposition}

\begin{proof}
For a proof, see \cite[section 3]{StSa}. 
\end{proof}

We also need to estimate the number of solutions of the corresponding {\it equation}
$$
\sum_{i,j=1}^n b_{ij}x_ix_j=0.
$$
This is carried out in the following proposition which is a consequence of a result in \cite{HB-circle}.

\begin{Proposition}[Heath-Brown] \label{Pytha}
Let $n\ge 3$ be an integer and $\varepsilon>0$. Then for any non-singular quadratic form $Q(x_1,...,x_n)=\sum_{i,j=1}^{n}b_{ij}x_ix_j$ with integer coefficients, we have  
$$
\sharp\{(x_1,x_2,...,x_{n})\in \mathbb{Z}^n  : Q(x_1,...,x_n)=0, \ |x_i|\le N\}= O(N^{n-2+\varepsilon}),
$$
as $N\rightarrow \infty$, where the implied constant depends on $Q$ and $\varepsilon$.
\end{Proposition}

\begin{proof}
Let $\Psi:\mathbb{R}\rightarrow \mathbb{R}_{\ge 0}$ be a smooth function satisfying $\Psi(t)=1$ if $|t|\le 1$ and $\Psi(t)=0$ if $|t|\ge 2$. Set 
$$
w(t_1,...,t_n):=\prod\limits_{i=1}^n \Psi(t_i)
$$  
for $(t_1,...,t_n)\in \mathbb{R}^n$. Then
$$
\sharp\{(x_1,x_2,...,x_{n})\in \mathbb{Z}^n  : Q(x_1,...,x_n)=0, \ |x_i|\le N\}\le \sum\limits_{\substack{(x_1,...,x_n)\in \mathbb{Z}^n\\ Q(x_1,...,x_n)=0}} w\left(\frac{x_1}{N},...,\frac{x_n}{N}\right). 
$$
From the asymptotic formula in \cite[Theorem 5]{HB-circle}, it follows that the right-hand side above is bounded by $O_{\varepsilon}(N^{n-2+\varepsilon})$ whenever $n\ge 3$, which implies the claim.  
\end{proof}

\section{Proof of Proposition \ref{P1}}
Let $m-r$ be even. Writing ${\bf{x}}=p^{(m-r)/2}{\bf{z}}+{\bf{y}}$ with ${\bf z}\in \left(\mathbb{Z}_p/p^{(m+r)/2}\mathbb{Z}_p\right)^{n}$ and ${\bf y}\in \left(\mathbb{Z}_p/p^{(m-r)/2}\mathbb{Z}_p\right)^{n}$, we obtain the following chain of equations.  
\begin{align*}
S_{\bf{\alpha}}(f,p^m)&=\sum_{\substack{{\bf{x}}\in (\mathbb{Z}_p/p^m\mathbb{Z}_{p})^n\\ {\bf{x}}\equiv {\bf{\alpha}}\bmod p}}e_{p^m}(f({\bf{x}}))\notag\\
&=\sum_{{\bf{z}}\in(\mathbb{Z}_p/p^{(m+r)/2}\mathbb{Z}_{p})^n}\sum_{\substack{{\bf{y}}\in(\mathbb{Z}_p/p^{(m-r)/2}\mathbb{Z}_{p})^n\\{\bf{y}}\equiv {\bf{\alpha}}\bmod p}}e_{p^{m}}(f({p^{(m-r)/2}\bf{z}}+{\bf y}))\notag\\
&=\sum_{\substack{{\bf{y}}\in(\mathbb{Z}_p/p^{(m-r)/2}\mathbb{Z}_{p})^n\\{\bf{y}}\equiv {\bf{\alpha}}\bmod p}}\sum_{{\bf{z}}\in(\mathbb{Z}_p/p^{(m+r)/2}\mathbb{Z}_{p})^n}e_{p^{m}}(f({\bf{y}})+p^{(m-r)/2}{\bf{z}}\cdot\nabla f({\bf{y}})+\overline{2}p^{m-r}{\bf{z}}^t\text{Hesse}(f)({\bf{y}}){\bf{z}})\notag\\
&=\sum_{\substack{{\bf{y}}\in(\mathbb{Z}_p/p^{(m-r)/2}\mathbb{Z}_{p})^n\\{\bf{y}}\equiv {\bf{\alpha}}\bmod p}}e_{p^{m}}(f({\bf{y}}))\sum_{{\bf{z}}\in(\mathbb{Z}_p/p^{(m+r)/2}\mathbb{Z}_{p})^n}e_{p^{m}}(p^{(m-r)/2}{\bf{z}}\cdot\nabla f({\bf{y}}))\notag\\
&=p^{nr}\sum_{\substack{{\bf{y}}\in(\mathbb{Z}_p/p^{(m-r)/2}\mathbb{Z}_{p})^n\\{\bf{y}}\equiv {\bf{\alpha}}\bmod p}}e_{p^{m}}(f({\bf{y}}))\sum_{{\bf{z}}\in(\mathbb{Z}_p/p^{(m-r)/2}\mathbb{Z}_{p})^n}e_{p^{(m-r)/2}}(p^{-r}{\bf{z}}\cdot\nabla f({\bf{y}}))\notag\\
&=p^{n(m+r)/2}\sum_{\substack{{\bf{y}}\in(\mathbb{Z}_p/p^{(m-r)/2}\mathbb{Z}_{p})^n\\{\bf{y}}\equiv {\bf{\alpha}}\bmod p\\p^{-r}\nabla f({\bf{y}})\equiv 0\bmod p^{(m-r)/2}}}e_{p^{m}}(f({\bf{y}})).
\end{align*}
Note that in the Taylor development in the third line, the $l$-th order term of the Taylor polynomial vanishes modulo $p^m$ for $l\geq 3$. Indeed, for $l\geq 3$ and $r\leq m-2$, the order of $p$ is equal to ${{\frac{(m-r)l}{2}}-\sum_{k=1}^{\infty}[\log_{p^k}l]}-(m-r)$, and 
\begin{align}\label{vanish1}
\frac{(m-r)l}{2}-\sum_{k=1}^{\infty}[\log_{p^k}l]-(m-r)&\geq (m-r)\left(\frac{l}{2}-1\right)-\sum_{k=1}^{\infty}\log_{p^k}l\notag\\
&= (m-r)\left(\frac{l}{2}-1\right)-\sum_{k\leq \log_pl}^{\infty}\frac{\log l}{k\log p}\notag\\
&\geq 2\left(\frac{l}{2}-1\right)-\frac{\log l}{\log p}\sum_{k\leq \log_pl}^{\infty}\frac{1}{k}\notag\\
&\geq \left(l-2\right)-\frac{\log l}{\log p}(1+\log\log_pl)\notag\\
& \geq 0.
\end{align}
The second order term involving the Hesse matrix obviously disappears as well due to the factor $p^{m-r}$.  

Next, we consider the case when $m-r$ is odd, i.e. $2|(m-r-1)$. The calculations are similar as above, and writing ${\bf{x}}=p^{(m-r+1)/2}{\bf{z}}+{\bf y}$ with ${\bf{z}}\in \left(\mathbb{Z}_p/p^{(m+r-1)/2}\mathbb{Z}_{p}\right)^n$ and ${\bf y}\in(\mathbb{Z}_p/p^{(m-r+1)/2}\mathbb{Z}_{p})^n$, we now obtain the following chain of equations. 
\begin{align*}
S_{\bf{\alpha}}(f,p^m)&=\sum_{\substack{{\bf{x}}\in (\mathbb{Z}_p/p^m\mathbb{Z}_{p})^n\\ {\bf{x}}\equiv {\bf{\alpha}}\bmod p}}e_{p^m}(f({\bf{x}}))\notag\\
&=\sum_{{\bf{z}}\in(\mathbb{Z}_p/p^{(m+r-1)/2}\mathbb{Z}_{p})^n}\sum_{\substack{{\bf{y}}\in(\mathbb{Z}_p/p^{(m-r+1)/2}\mathbb{Z}_{p})^n\\{\bf{y}}\equiv {\bf{\alpha}}\bmod p}}e_{p^{m}}(f({p^{(m-r+1)/2}\bf{z}}+{\bf y}))\notag\\
&=\sum_{\substack{{\bf{y}}\in(\mathbb{Z}_p/p^{(m-r+1)/2}\mathbb{Z}_{p})^n\\{\bf{y}}\equiv {\bf{\alpha}}\bmod p}}\sum_{{\bf{z}}\in(\mathbb{Z}_p/p^{(m+r-1)/2}\mathbb{Z}_{p})^n}e_{p^{m}}(f({\bf{y}})+p^{(m-r+1)/2}{\bf{z}}\cdot\nabla f({\bf{y}}))\notag\\
&=p^{nr}\sum_{\substack{{\bf{y}}\in(\mathbb{Z}_p/p^{(m-r+1)/2}\mathbb{Z}_{p})^n\\{\bf{y}}\equiv {\bf{\alpha}}\bmod p}}e_{p^{m}}(f({\bf{y}}))\sum_{{\bf{z}}\in(\mathbb{Z}_p/p^{(m-r-1)/2}\mathbb{Z}_{p})^n}e_{p^{(m-r-1)/2}}(p^{-r}{\bf{z}}\cdot\nabla f({\bf{y}}))\notag\\
&=p^{(m+r-1)n/2}\sum_{\substack{{\bf{y}}\in(\mathbb{Z}_p/p^{(m-r+1)/2}\mathbb{Z}_{p})^n\\{\bf{y}}\equiv {\bf{\alpha}}\bmod p\\p^{-r}\nabla f({\bf{y}})\equiv 0\bmod p^{(m-r-1)/2}}}e_{p^{m}}(f({\bf{y}})).
\end{align*}
Similarly as before, the the $l$-th order terms with $l\ge 2$ in the Taylor series expansion will vanish. This proves Proposition \ref{P1}. 
\begin{flushright}
$\Box$
\end{flushright}

\section{Proof of Theorem \ref{MT1}}
Part (i) follows trivially from Proposition \ref{P1}. For the proof of part (ii),
assume that Hesse$(f)(\alpha)$ is non-singular, i.e. det(Hesse$(f)(\alpha))\not\equiv 0 \bmod{p}$. By the generalized Hensel lemma, ${\bf \alpha}$ has a unique lifting ${\bf{\alpha^*}}$ to a solution of the congruence $p^{-r}\nabla f({\bf{y}})\equiv 0\bmod p^m$. If $m-r$ is even, then from \eqref{basiceq}, we have 
\begin{align*}
S_{\bf{\alpha}}(f,p^m)=p^{n(m+r)/2}e_{p^{m}}(f({\bf{\alpha^*}})).
\end{align*} 
Similarly, if $m-r$ is odd, then from \eqref{basiceq}, we obtain
\begin{align*}
S_{\bf{\alpha}}(f,p^m)&=p^{(m+r-1)n/2}\sum_{\substack{{\bf{u}}\in(\mathbb{Z}_p/p\mathbb{Z}_{p})^n}}e_{p^{m}}(f({\bf{\alpha}^*}+p^{(m-r-1)/2}{\bf{u}}))\\
&=p^{(m+r-1)n/2}\sum_{\substack{{\bf{u}}\in(\mathbb{Z}_p/p\mathbb{Z}_{p})^n}}e_{p^{m}}(f({\bf{\alpha}^*})+p^{(m-r-1)/2}{\bf{u}}\cdot \nabla f({\bf{\alpha}^*})+p^{m-r-1}\overline{2}{\bf{u}}^t\text{Hesse}(f)({\bf{\alpha}^*}){\bf{u}})\\
&=p^{(m+r-1)n/2}e_{p^{m}}(f({\bf{\alpha}^*}))\sum_{\substack{{\bf{u}}\in(\mathbb{Z}_p/p\mathbb{Z}_{p})^n}}e_{p}(\overline{2}p^{-r}{\bf{u}}^t\text{Hesse}(f)({\bf{\alpha}^*}){\bf{u}}).
\end{align*} 
Take $A=\overline{2}p^{-r}\text{Hesse}(f)({\bf{\alpha}^*})$, which is a symmetric matrix. Hence there exists an orthogonal matrix over $\mathbb{F}_p$, say $P,$ such that
\begin{align*}
PAP^t=D=\begin{bmatrix}
\lambda_1&0&...&0\\
0&\lambda_2&...&0\\
\vdots &\vdots&\cdots &\vdots\\
0&0&...&\lambda_n
\end{bmatrix}
\end{align*} 
over $\mathbb{F}_p$, where $\lambda_1,...,\lambda_n$ are the eigenvalues of $A$ over $\mathbb{F}_p$. 
Let 
\begin{align*}
P=\begin{bmatrix}
R_1\\
R_2\\
\vdots \\
R_n
\end{bmatrix},
\end{align*}
where $R_i$ are the row vectors of $P$. Now ${\bf u}^tA{\bf u}={\bf{u}}^tP^tDP{\bf{u}}=\lambda_1(R_1{\bf{u}})^2+\lambda_2(R_2{\bf{u}})^2+\cdots+\lambda_n(R_n{\bf{u}})^2$.
As $P$ is invertible, the map ${\bf{u}}\mapsto P{\bf u}$ is a bijection on $(\mathbb{Z}_p/p\mathbb{Z}_{p})^n$. Hence we get 
\begin{align*}
\sum_{\substack{{\bf{u}}\in(\mathbb{Z}_p/p\mathbb{Z}_{p})^n}}e_{p}(\overline{2}p^{-r}{\bf{u}}^t\text{Hesse}(f)({\bf{\alpha}^*}){\bf{u}})&=\sum_{\substack{{\bf{u}}\in(\mathbb{Z}_p/p\mathbb{Z}_{p})^n}}e_{p}(\lambda_1(R_1{\bf{u}})^2+\lambda_2(R_2{\bf{u}})^2+\cdots+\lambda_n(R_n{\bf{u}})^2)\\
&=G_p(\lambda_1)\cdots G_p(\lambda_n)\\
&=\left(\frac{\lambda_1\cdots\lambda_n}{p}\right)G_p(1)^n.
\end{align*}
This implies 
\begin{align*}
S_{\bf{\alpha}}(f,p^m)=\left(\frac{\text{det}(\overline{2}p^{-r}\text{Hesse}(f)({\bf{\alpha}^*}))}{p}\right)p^{(m+r-1)n/2}e_{p^{m}}(f({\bf{\alpha}^*}))G_p(1)^n, \quad \mbox{if } m-r \mbox{ is odd,}
\end{align*}
which completes the proof of part (ii). \begin{flushright} $\Box$ \end{flushright}

\section{Proof of Theorem \ref{MT2}}

\subsection{Multiple Poisson summation} 
Using orthogonality properties of character sums, we start by writing 
\begin{equation*}
\begin{split}
T= & \sum\limits_{\substack{(x_1,x_2,..,x_{n})\in \mathbb{Z}^{n}\\(x_1,...,x_n)\not\equiv (0,...,0)\bmod p\\  \sum_{i,j=1}^n a_{ij}x_ix_j\equiv 0 \bmod{p^m}}}\prod_{i=1}^{n} \Phi\left(\frac{x_i-x_{0,i}}{N}\right)\\
&=\frac{1}{p^m} \sum\limits_{\substack{(x_1,x_2,..,x_{n})\in \mathbb{Z}^{n}\\(x_1,...,x_n)\not\equiv (0,...,0)\bmod p}}\prod_{i=1}^{n} \Phi\left(\frac{x_i-x_{0,i}}{N}\right)\sum_{h=1}^{p^m}e_{p^m}\left(h \sum_{i,j=1}^n a_{ij}x_ix_j\right).
\end{split}
\end{equation*}
Setting $x_i=y_i+k_ip^m$, where $k_i\in \mathbb{Z}$ and $1\leq y_i\leq p^m$, we apply Proposition \ref{Poisson} to obtain
\begin{equation*}
T= \frac{N^{n}}{p^{m(n+1)}}\sum\limits_{(k_1,..,k_{n})\in \mathbb{Z}^{n}} \prod_{i=1}^{n}\hat{\Phi}\left(\frac{k_iN}{p^m}\right)\sum_{h=1}^{p^m}\sum_{\substack{{y_1},...,{y_{n}}=1\\(y_1,...,y_n)\not\equiv (0,...,0)\bmod p}}^{p^m} e_{p^m}\left(h \sum_{i,j=1}^n a_{ij}y_iy_j+\sum_{i=1}^{n}k_i(y_i-x_{0,i})\right).
\end{equation*}
We decompose $T$ into two parts, 
\begin{equation} \label{divide}
T=T_0+U,
\end{equation}
where $T_0$ is the main term contribution corresponding to $(k_1,...,k_{n})=(0,...,0)$. Hence,
\begin{equation*}
T_0= \hat{\Phi}(0)^{n}\cdot \frac{N^{n}}{p^{m}}\cdot\frac{\#\mathcal{A}}{p^{m(n-1)}},
\end{equation*}
where $\mathcal{A}$ is defined as in Theorem \ref{MT2}. 
In the following, we write $(\mathbb{Z}^{n})^{\ast}=\mathbb{Z}^{n}\setminus (0,0,...,0).$
\subsection{Evaluation of exponential sums}
We write $Q(x_1,...,x_n)=X^tAX$, where $A$ is the corresponding symmetric matrix for the quadratic form $Q$ and $X=(x_1,...,x_n)^t$.  Now we look at the error contribution
\begin{align} \label{errorcont}
U&= \frac{N^{n}}{p^{m(n+1)}}\sum\limits_{(k_1,..,k_{n})\in (\mathbb{Z}^{n})^{*}} \prod_{i=1}^{n}\hat{\Phi}\left(\frac{k_iN}{p^m}\right)e_{p^m}\left(-\sum_{i=1}^{n}k_ix_{0,i}\right)\sum_{h=1}^{p^m} E(k_1,...,k_{n},h;p^m),
\end{align}
with 
\begin{equation*}
E(k_1,...,k_{n},h;p^m):=\sum_{\substack{{y_1},...,{y_{n}}=1\\(y_1,...,y_n)\not\equiv (0,...,0)\bmod p}}^{p^m} e_{p^m}\left(h \sum_{i,j=1}^n a_{ij}y_iy_j+\sum_{i=1}^{n}k_iy_i\right).
\end{equation*}
Assume that
$$
(k_1,...,k_{n},p^m)=p^r.
$$
Take 
\begin{align}\label{jan-2}
    f_{k_1,...,k_{n},h}(y_1,y_2,...,y_n):=h \sum_{i,j=1}^n a_{ij}y_iy_j+\sum_{i=1}^{n}k_iy_i.
\end{align}
The contributions of $r=m-1$ and $r=m$ to the right-hand side of \eqref{errorcont} are $O_{\varepsilon}(1)$ if $N\ge p^{m\varepsilon}$ by the rapid decay of $\hat\Phi$ as  $(k_1,...,k_{n})\neq(0,...,0)$. We fix $ l_i=k_i/p^r$. From \eqref{jan-2} we calculate 
\begin{align*}
\nabla f_{k_1,...,k_{n},h}=\left(h\frac{\partial Q}{\partial y_1}+p^rl_1,...,h\frac{\partial Q}{\partial y_n}+p^rl_n\right).
\end{align*}  
In the following, we write $f=f_{k_1,...,k_n,h}$ for simplicity. 
If $\text{\rm ord}_p(h)> r$ then $\text{\rm ord}_p(\nabla(f))= r$. So $p^{-r}\nabla(f)\equiv (l_1,...l_n) \bmod p, $ where $(l_1,...,l_n)\not \equiv (0,...,0)\bmod p$. Hence there is no critical point of $f_{k_1,...,k_{n},h}$. Now using $(i)$  of Theorem \ref{MT1}, we have $E(k_1,...,k_{n},h;p^m)=0$. On the other hand if $\text{\rm ord}_p(h)<r$ then $\text{\rm ord}_p(\nabla(f))= \text{\rm ord}_p(h)$. Hence 
$
p^{-\text{\rm ord}_p(h)}\nabla(f)\equiv {\boldsymbol{0}} \bmod p,
$ implies $AY\equiv {\boldsymbol{0}}\bmod p$ for $Y=(y_1 ~y_2~\cdots~y_n)^t$. Since $A$ is non-singular, $Y=(0,...,0)$ is the only critical point in this case. Hence $E(k_1,...,k_{n},h;p^m)=0.$ Thus from now onward we consider $\text{\rm ord}_p(h)=r$. Therefore $p^{\text{\rm ord}_p(\nabla(f))}=p^r=(h, p^m)$ and we write $ l=h/p^r$. So if $N\ge p^{m\varepsilon}$, we have 
\begin{align}\label{jan-3}
    U= \frac{N^{n}}{p^{m(n+1)}}\sum_{r=0}^{m-2}\sum\limits_{\substack{(l_1,..,l_{n})\in \mathbb{Z}^{n}\\(l_1,...,l_{n})\not\equiv(0,...,0)\bmod{p}}} \prod_{i=1}^{n}\hat{\Phi}\left(\frac{l_iN}{p^{m-r}}\right)e_{p^m}\left(-\sum_{i=1}^{n}l_ip^rx_{0,i}\right)\sum_{\substack{l=1\\(l,p)=1}}^{p^{m-r}}E(l_1p^r,...,l_{n}p^r,lp^r;p^m)+O_{\varepsilon}(1).
\end{align}
  We calculate that 
  \begin{align}\label{grad}
     \frac{\partial f_{l_1p^r,...,l_np^r, lp^r}}{\partial y_i}(y_1,y_2,...,y_n)=p^{r}\left(l\frac{\partial{Q}}{\partial{y_i}}+l_i\right), ~ \forall ~i.
  \end{align}
In the following, we will apply Proposition \ref{P1} to evaluate the exponential sums $E(l_1p^r,...,l_{n}p^r;p^m)$. Now consider the case when $m-r$ is even. Note that $A$ is the symmetric matrix associated with the quadratic form $Q(y_1,...,y_n)=\sum_{i,,j=1}^na_{ij}y_iy_j$. Hence $\frac{\partial Q}{\partial y_j}=2(a_{j1}y_1+a_{j2}y_2+\cdots+a_{jn}y_n)=2(a_{j1}~ a_{j2}~ \cdots~ a_{jn}) Y$. Therefore
$
p^{-r}\nabla f_{k_1,...,k_{n},h}(y_1,...,y_n)\equiv \boldsymbol{0} \bmod p^{m-r}
$
implies 
\begin{align*}
2lAY+L\equiv \boldsymbol{0} \bmod p^{m-r}.
\end{align*}
As $A$ is non-singular modulo $p$, we have
\begin{align}\label{jan-1}
Y\equiv -\overline{2l}\overline{\text{det}(A)}\text{adj}(A) L\bmod p^{m-r},
\end{align}
where $Y=(y_1 ~y_2~\cdots~y_n)^t$ and $L=(l_1~l_2~\cdots~ l_n)^t$. This is the only critical point for $f_{k_1,...,k_{n},h}(y_1,y_2,...,y_n).$ Now putting the values of $y_i$ in \eqref{jan-2}, we get
\begin{align*}
-\overline{2l\text{det}(A)}(1-\overline{2})L^t \text{adj}(A)^t L\equiv 0 \bmod  p^{m-r},
\end{align*}
where we take the inverse modulo $p^{m-r}$.
Hence  when $m-r$ is even, then using Proposition 1, we get 
\begin{align*}
E(l_1p^r,...,l_{n}p^r, lp^r;p^m):&=\sum_{\substack{{y_1},...,{y_{n}}=1}}^{p^m} e_{p^m}\left(lp^r \sum_{i,j=1}^n a_{ij}y_iy_j+p^r\sum_{i=1}^{n}l_iy_i\right)\\ &=p^{(m+r)n/2}e_{p^{m-r}}\left(C\cdot \overline{l}\cdot Q'(L)\right),
\end{align*}
where $C=-\overline{2\text{det}(A)}(1-\overline{2})$ and $Q'(L)=L^t \text{adj}(A)^tL$ which we refer to as the dual of the initial quadratic form.
\par Now we consider the case when $m-r$ is odd. We calculate that 
\begin{align*}
\text{Hesse}(f_{l_1p^r,...,l_{n}p^r, lp^r})(y_1,...,y_n)=2p^rlA.
\end{align*}
 Then using $(ii)$ of Theorem \ref{MT1}, we get 
\begin{align*}
E(l_1p^r,...,l_{n}p^r,lp^r;p^m)=\left(\frac{\text{det}(2lA)}{p}\right)p^{(m+r-1)n/2}e_{p^{m}}(f({\bf{y}}))G_p(1)^n,
\end{align*}
where ${\bf y}=Y^t$.
Similarly as in \eqref{jan-1}, we get 
\begin{align*}
E(l_1p^r,...,l_{n}p^r,lp^r;p^m)=\left(\frac{\text{det}(2A)}{p}\right)\left(\frac{l^n}{p}\right)p^{(m+r-1)n/2}e_{p^{m-r}}\left(C\cdot \overline{l}\cdot Q'(L)\right)G_p(1)^n,
\end{align*} 
where $Q'(L)$ and $C$ are as above. So combining the even and odd cases we can write $E(l_1p^r,...,l_{n}p^r,lp^r;p^m)$ as 
\begin{align}\label{jan-4}
E(l_1p^r,...,l_{n}p^r,lp^r;p^m)=\begin{cases}
p^{(m+r)n/2}e_{p^{m-r}}\left(C\cdot \overline{l}\cdot Q'(L)\right), &m-r ~\text{even};\\
p^{(m+r-1)n/2}\left(\frac{\text{det}(2A)}{p}\right)\left(\frac{l^n}{p}\right)e_{p^{m-r}}\left(C\cdot \overline{l}\cdot Q'(L)\right)G_p(1)^n, &m-r ~\text{odd}.
\end{cases}
\end{align}
 Observe that for any $a\in \mathbb{Z}$, we have
  \begin{align*}
    \sum_{\substack{x=1\\ x\not \equiv 0\bmod p}}^{p^k} e_{p^k}(xa)=0 \quad \mbox{ if } p^{k-1}\nmid a
  \end{align*}
and 
\begin{align*}
    \sum_{\substack{x=1\\ x\not \equiv 0\bmod p}}^{p^k} \left(\frac{x}{p}\right)e_{p^k}(xa)&=  
\sum\limits_{b=1}^{p-1} \left(\frac{b}{p}\right)\sum_{\substack{x=1\\ x\equiv b\bmod p}}^{p^k} e_{p^k}(xa)\\
&=\sum\limits_{b=1}^{p-1} \left(\frac{b}{p}\right)\sum_{\substack{y=1}}^{p^{k-1}} e_{p^k}((py+b)a)\\
&=\sum\limits_{b=1}^{p-1} \left(\frac{b}{p}\right)e_{p^k}(ba)\sum_{\substack{y=1}}^{p^{k-1}} e_{p^{k-1}}(ya)\\
&=0 \quad \mbox{ if } p^{k-1}\nmid a.
  \end{align*}
Using the above equalities, we deduce from \eqref{jan-4} that
 \begin{align*}
 \sum_{\substack{l=1\\(l,p)=1}}^{p^{m-r}}E(l_1p^r,...,l_{n}p^r,lp^r;p^m)=0
 \end{align*}
 whenever $p^{m-r-1}\nmid Q'(L)$. Otherwise, a trivial estimate gives
 \begin{align*}
  \sum_{\substack{l=1\\(l,p)=1}}^{p^{m-r}} |E(l_1p^r,...,l_{n}p^r,lp^r;p^m)|\leq p^{(m+r)n/2+(m-r)}.
 \end{align*}

\subsection{Error estimation}
The error $U$ can be written as 
\begin{align}\label{jan-6}
 U&= \frac{N^{n}}{p^{m(n+1)}}\sum_{r=0}^{m-2}\sum\limits_{\substack{(l_1,..,l_{n})\in \mathbb{Z}^{n}\\(l_1,...,l_{n})\not\equiv(0,...,0)\bmod{p}}} \prod_{i=1}^{n}\hat{\Phi}\left(\frac{l_iN}{p^{m-r}}\right)e_{p^m}\left(-\sum_{i=1}^{n}l_ip^rx_{0,i}\right)\sum_{\substack{l=1\\(l,p)=1}}^{p^{m-r}}E(l_1p^r,...,l_{n}p^r,lp^r;p^m)+O_{\varepsilon}(1)\notag\\
 &\ll\frac{N^{n}}{p^{m(n+1)}}\sum_{r=0}^{m-2}p^{(m+r)n/2+(m-r)}\sum\limits_{\substack{(l_1,..,l_{n})\in \mathbb{Z}^{n}\\(l_1,...,l_{n})\not\equiv(0,...,0)\bmod{p}\\Q'(L)\equiv 0 \bmod p^{m-r-1}}} \prod_{i=1}^{n}\hat{\Phi}\left(\frac{l_iN}{p^{m-r}}\right)+O_{\varepsilon}(1)
 .\end{align}
Set
\begin{align*}
L_r:=p^{m-r+m\varepsilon}N^{-1}.
\end{align*}
Now using the rapid decay of $\hat\Phi$, the above summations over $l_1,...,l_{n}$ in \eqref{jan-6} can be cut off at $|l_1|,...,|l_{n}|\le L_r$ at the cost of a negligible error if $m$ is large enough. Using \eqref{jan-6}, we obtain
 \begin{equation*} 
\begin{split}
U\ll & \frac{N^{n}}{p^{m(n+1)}}\sum\limits_{r=0}^{m-2}  p^{n(m+r)/2}p^{m-r} \sum\limits_{\substack{(l_1,..,l_{n})\in \mathbb{Z}^{n}\\(l_1,...,l_{n})\not\equiv(0,...,0)\bmod{p}\\|l_1|,...,|l_{n}|\le L_r\\Q'(L)\equiv 0 \bmod p^{m-r-1}}} 1 +O_{\varepsilon}(1).
\end{split}
\end{equation*}


Let $Q'(L)=L^t \text{adj}(A)^t L=\sum_{i,j}b_{i,j}l_il_j$, where the $b_{ij}$'s are only depending on $a_{ij}$'s. Consider $M=\text{max} \{|b_{i,j}|, 1\leq i,j\leq n\}$. Now if $L_r< \sqrt{p^{m-r-1}/
(Mn^2)}$, then the congruence above can be replaced by the equation $Q'(L)=0$. Certainly, this is the case if $N\ge p^{m/2+4m\varepsilon}$ and $m$ is large enough. Hence in this case, we have
\begin{equation*} 
\begin{split}
U\ll & \frac{N^{n}}{p^{m(n+1)}}\sum\limits_{r=0}^{m-2}  p^{n(m+r)/2}p^{m-r} \sum\limits_{\substack{(l_1,..,l_{n})\in \mathbb{Z}^{n}\\(l_1,...,l_{n})\not\equiv(0,...,0)\bmod{p}\\|l_1|,...,|l_{n}|\le L_r\\Q'(L)= 0 }} 1 +O_{\varepsilon}(1).
\end{split}
\end{equation*}
Finally, we apply Proposition \ref{Pytha} to bound $U$ by 
\begin{align*}
    U\ll & \frac{N^{n}}{p^{m(n+1)}}\sum\limits_{r=0}^{m-2}  p^{n(m+r)/2}p^{m-r} L_r^{n-2+\varepsilon}+O_{\varepsilon}(1)
\ll N^2p^{m(n-4)/2+mn\varepsilon}
\end{align*} 
if $\varepsilon\le 1$. 
This needs to be compared to the main term which is of size
$$
T_0\asymp \frac{N^{n}}{p^m}.
$$
Hence, if $N\ge p^{(1/2+3\varepsilon)m}$, then $U=o(T_0)$, which completes the proof of Theorem \ref{MT2} upon changing $3\varepsilon$ into $\varepsilon$.

\end{document}